\documentclass[11pt, reqno]{amsart}
\usepackage{amsmath, amsfonts, amsthm, amssymb, multicol, mathtools, dsfont, mathrsfs}
\usepackage{graphicx}
\usepackage{float, hyperref}
\usepackage{scalerel,stackengine, subcaption}
\usepackage[usenames,dvipsnames,x11names]{xcolor}
\usepackage{enumitem}
\usepackage{pgfplots}
\usepgfplotslibrary{fillbetween}
\pgfplotsset{width=10cm,compat=1.9}
\usepackage[square,sort,comma,numbers]{natbib}
\allowdisplaybreaks

\makeatletter
\g@addto@macro\bfseries{\boldmath}
\makeatother

\makeatletter
\def\@setauthors{%
  \begingroup
  \def\thanks{\protect\thanks@warning}%
  \trivlist
  \centering\footnotesize \@topsep30\p@\relax
  \advance\@topsep by -\baselineskip
  \item\relax
  \author@andify\authors
  \def\\{\protect\linebreak}

  \normalsize\lowercase{\authors}%
  
	\ifx\@empty\contribs
  \else
    ,\penalty-3 \space \@setcontribs
    \@closetoccontribs
  \fi
  \endtrivlist
  \endgroup
}
\def\@settitle{\begin{center}
\LARGE\lowercase{\@title}
  \end{center}%
}
\makeatother
\newcommand{\authoremail}[1]{\email{\href{mailto:#1}{\color{lightblue}{#1}}}}
\newcommand{\authoraddress}[1]{\address{\normalfont{#1}}}

\numberwithin{equation}{section}

\newtheorem{thm}{Theorem}[section]
\newtheorem{lma}[thm]{Lemma}
\newtheorem{cor}[thm]{Corollary}
\newtheorem{defn}[thm]{Definition}

\renewcommand{\epsilon}{\varepsilon}
\newcommand{\eps}{\varepsilon}

\renewcommand{\geq}{\geqslant}
\renewcommand{\leq}{\leqslant}

\newcommand{\hd}{\dim_{\textup{H}}}

\makeatletter
\DeclareRobustCommand\widecheck[1]{{\mathpalette\@widecheck{#1}}}
\def\@widecheck#1#2{%
    \setbox\z@\hbox{\m@th$#1#2$}%
    \setbox\tw@\hbox{\m@th$#1%
       \widehat{%
          \vrule\@width\z@\@height\ht\z@
          \vrule\@height\z@\@width\wd\z@}$}%
    \dp\tw@-\ht\z@
    \@tempdima\ht\z@ \advance\@tempdima2\ht\tw@ \divide\@tempdima\thr@@
    \setbox\tw@\hbox{%
       \raise\@tempdima\hbox{\scalebox{1}[-1]{\lower\@tempdima\box
\tw@}}}%
    {\ooalign{\box\tw@ \cr \box\z@}}}
\makeatother

\stackMath
\newcommand\reallywidehat[1]{%
\savestack{\tmpbox}{\stretchto{%
  \scaleto{%
    \scalerel*[\widthof{\ensuremath{#1}}]{\kern.1pt\mathchar"0362\kern.1pt}%
    {\rule{0ex}{\textheight}}
  }{\textheight}%
}{2.4ex}}%
\stackon[-6.9pt]{#1}{\tmpbox}%
}
\definecolor{lightblue}{HTML}{2B77A4}
\colorlet{plotblue}{LightSkyBlue3!80}
\definecolor{darkred}{HTML}{9E0D0D}
\definecolor{purp}{HTML}{d603a9}
\definecolor{dartmouthgreen}{HTML}{00A64F}
\definecolor{Junglegreen}{HTML}{00A99A}
\definecolor{yellowcolour}{HTML}{f07c02}
\hypersetup{
	colorlinks=true,
	linkcolor=darkred,
	urlcolor=darkred,
	citecolor=lightblue
}
\title{Exceptional projections in finite fields:\\
Fourier analytic bounds  and  incidence geometry}

\usepackage{fancyhdr}
\author{\large{Jonathan M. Fraser}}
\authoraddress{Jonathan M. Fraser, University of St Andrews, Scotland}
\authoremail{jmf32@st-andrews.ac.uk}
\thanks{JMF was  financially supported by a  \emph{Leverhulme Trust Research Project Grant} (RPG-2023-281)  and an \emph{EPSRC Standard Grant} (EP/Y029550/1).}

 \author{\large{Firdavs Rakhmonov}}
\authoraddress{Firdavs Rakhmonov, University of St Andrews, Scotland}
\authoremail{fr52@st-andrews.ac.uk}
\thanks{FR was  financially supported by a  \emph{Leverhulme Trust Research Project Grant} (RPG-2023-281).}

\date{}

\begin{document}
\thispagestyle{empty}

\begin{abstract}
We consider the problem of bounding the number of exceptional projections (projections which are smaller than typical)  of a subset of a vector space over a finite field onto subspaces.  We establish bounds  that depend on $L^p$ estimates for the Fourier transform, improving various known bounds for sets with sufficiently good Fourier analytic properties.  The special case $p=2$ recovers a recent result of Bright and Gan (following Chen), which established the finite field analogue of Peres--Schlag's bounds from the continuous setting.

We prove several auxiliary results of independent interest, including a character sum identity for subspaces (solving a problem of Chen) and a full generalization of Plancherel's theorem for subspaces. These auxiliary results also have applications in affine incidence geometry, that is, the problem of estimating  the number of incidences between a set of points and a set of affine $k$-planes. We present a novel and direct proof of a well-known result in this area that avoids the use of spectral graph theory, and we provide simple examples demonstrating that these estimates are sharp up to constants.

\end{abstract}
\maketitle
\tableofcontents

\section{Introduction}

Marstrand's projection theorem is one of the most celebrated and fundamental results in fractal geometry.  It was proved in 1954 but later generalised by Kaufmann (who gave a simpler Fourier analytic proof) and Mattila (who extended it to higher dimensions); see \cite{mattila75, mattila}. It states that for a Borel set $E \subseteq \mathbb{R}^n$ with Hausdorff dimension $\hd E$, the Hausdorff dimension of the orthogonal projection of $E$ onto almost all $k$-dimensional subspaces is $\min\{\hd E,k\}$. Here, `almost all' is taken with respect to the Grassmannian measure supported on the Grassmannian manifold $G(k,n)$, which consists of all $k$-dimensional linear subspaces of $\mathbb{R}^n$. There has been an enormous amount of work studying the exceptional set in Marstrand's theorem. Due to work of Mattila, Falconer, Bourgain, Peres--Schlag, and others, we know the following refinement of Marstrand's theorem, stated in a form due to Mattila \cite{mattila75} and Peres--Schlag \cite{peres}. For a Borel set $E \subseteq \mathbb{R}^n$,
\begin{equation} \label{marstrand}
\hd \{ V \in G(k,n) : \hd \pi_V(E) \leq u \} \leq k(n-k) + u-\max\{\hd E ,k\}
\end{equation}
for all $0 \leq u < \min\{\hd E,k\}$ such that the right-hand side is non-negative. Here, $\pi_V(E)$ denotes the orthogonal projection of $E$ onto $V$. Moreover,  $k(n-k)$ is the dimension of the Grassmannian $G(k,n)$, so these exceptional sets are not only of zero Grassmannian measure but also of smaller dimension.

There is significant interest in analogues of these results in vector spaces over finite fields \cite{chen, bright,lundphamvinh}. This continues a productive line of research in which classical problems from geometric measure theory are adapted to the finite field setting. Famous examples include the Kakeya problem \cite{dvir, wolff},  the distance set problem \cite{iosevich}, and Szemer\'edi--Trotter type incidence geometry \cite{bennett, vinh}. 

Let $q$ be a power of a prime, and let $\mathbb{F}_q$ denote the finite field with $q$ elements. Let $\mathbb{F}_q^n$ be the $n$-dimensional vector space over $\mathbb{F}_q$, where $n \geq 1$.  We refer the reader to \cite{lidl} for the basic theory of finite fields. The appropriate analogue of \eqref{marstrand} in the finite field setting is the following:  For $E \subseteq \mathbb{F}_q^n$, 
\begin{equation} \label{marstrandfinite}
    |\{V\in G(k,n): |\pi_V(E)|\leq u\}|\lesssim  \frac{q^{k(n-k)} u }{\max\{ |E|, q^{k}\}}
\end{equation}
for all $0 \leq u \leq q^{-\eps} \min\{ |E|, q^k\}$ for some $\eps>0$. Here, $G(k, n)$ again denotes the set of all $k$-dimensional linear subspaces of $\mathbb{F}_q^n$, which, in this case, is a finite set. Moreover, $\pi_V(E)$ denotes the projection of $E \subseteq \mathbb{F}_q^n$ onto the subspace $V$ of $\mathbb{F}_q^n$; see Definition \ref{projdef} for the precise definition. 

In the above and throughout, the notation $A\lesssim B$ signifies that $A\leq C B$ for some absolute constant $C>0$. Similarly, we write $A\lesssim_\lambda B$ to indicate the existence of a constant $C(\lambda)>0$, depending on the parameter $\lambda$, such that $A\leq C(\lambda) B$. Furthermore, we use $A\approx B $ to denote that both $A \lesssim B$ and $B \lesssim A$ hold, i.e., there exist universal positive constants $C_1$ and $C_2$ such that $C_1 A\leq B\leq C_2A$.  We also use $|X|$ to denote the cardinality of a finite set $X$ and $\mathbb{N}_0$ to denote the set of non-negative integers. Additionally, we will use standard asymptotic notations such as $f=o(g), f=O(g),$ and $f\sim g$, which mean that (for non-negative functions $f,g$ of $q$)  $f/g\to 0$, $f/g$ is bounded by some absolute constant, and $f/g\to 1$, respectively.

Chen \cite[Theorem 1.2]{chen} established \eqref{marstrandfinite} in the case of prime fields, and Bright and Gan extended this to arbitrary fields; see \cite[Theorem 6 and equation (13)]{bright}.  Our main result   (Theorem \ref{main projection theorem}) is a strengthening of \eqref{marstrandfinite}, which takes into account $L^p$ properties of the Fourier transform. This strategy follows a general approach introduced in \cite{fraserfinite} and also constitutes a finite field analogue of certain results in \cite{ana}, which address the Euclidean case.  The special case $p=2$ recovers \eqref{marstrandfinite} and provides an alternative proof to that of Bright--Gan  (see Corollary \ref{main cor}).

Our argument follows Chen's  approach, and in particular, a key step is establishing the analogue of \cite[Lemma 2.4]{chen} in the non-prime case.  This result involves certain character sums associated with subspaces of $\mathbb{F}_q^n$ and is of independent interest (see Lemma \ref{sum of characters}). It resolves an explicit problem posed by Chen, whose resolution was also conjectured by Bright and Gan. Further auxiliary results of independent interest include a full generalization of Plancherel's theorem for subspaces of $\mathbb{F}_q^n$ (Theorem \ref{complete Plancherel on subspaces}) and certain moment identities for incidences, generalizing a result of Murphy, Petridis, and Chen (Lemma \ref{0-th moment}).

Incidence geometry, and in particular the study of incidences between a set of points and a set of lines, is an active and central topic in combinatorics, with applications and connections to many problems.  We are interested in incidence geometry in vector spaces over finite fields, which has attracted sustained attention over many years and also found applications to several problems, including sum-product theory, the distance set problem, and many others. While working on exceptional projections, we discovered that many of our auxiliary results could also be used to estimate the number of incidences between a set of points and a set of $k$-planes in $\mathbb{F}_q^n$. This problem has an interesting history, and there are many results in the literature that obtain incidence bounds of this type by appealing to spectral graph theory. Notable examples include  Vinh's result on point--line incidences \cite{vinh}, as well as later work by Lund and Saraf \cite{lund} and also Bennett, Iosevich, and Pakianathan \cite{bennett}. In particular, we know the following result.  If $I(E,\mathcal{A})$ denotes the number of incidences between a set of points $E \subseteq \mathbb{F}_q^n$ and a set of affine $k$-planes $\mathcal{A}$ in $\mathbb{F}_q^n$, then
\begin{equation} \label{knownit}
\left| I(E,\mathcal{A}) - \frac{|E||\mathcal{A}|}{q^{n-k}} \right| \leq \left(q^{k(n-k)}|E||\mathcal{A}|\right)^{\frac{1}{2}} (1+o(1)).
\end{equation}

Despite there being a lot of interest in this area, it was only recently rediscovered that this result, in fact, goes back to the 1980s in work of Haemers \cite{haemers}; see also \cite{dvirda}. As with the results mentioned above, Haemers results on incidences are obtained via spectral graph theory, in particular, by defining an appropriate graph associated with the points and planes and then appealing to the expander mixing lemma. We provide a new and self-contained proof of \eqref{knownit} that does not use spectral graph theory (see Theorem \ref{incidence theorem}). We also provide simple examples showing that \eqref{knownit} is the best possible result, up to constants, in all cases as a function of $|E||\mathcal{A}|$. 

In the final section \ref{biabh}, we tie the paper together by providing an alternative proof of \eqref{marstrandfinite} as a direct application of the incidence bound \eqref{knownit} (see Corollary \ref{projfromit}). On a heuristic level, the connection between these results is clear: for a projection onto a particular $V \in G(k,n)$ to be small, the fibers should be large, which corresponds to many incidences with a particular set of affine $(n-k)$-planes orthogonal to $V$. We note that our main theorem (Theorem \ref{main projection theorem}),  which improves \eqref{marstrandfinite} using $L^p$ estimates for the Fourier transform,  does \emph{not} appear to follow from incidence geometry.

\section{Background and preliminaries}

\subsection{Basics of discrete Fourier analysis} We use the machinery of discrete Fourier analysis to prove several of our results. In this short section, we provide an overview of the key concepts.

For a function $f : \mathbb{F}_q^n \to \mathbb{C}$, we define its \textit{Fourier transform} $\widehat{f} : \mathbb{F}_q^n \to \mathbb{C}$ as follows:
\begin{equation}
\label{FT}
    \widehat{f}(m)\coloneqq q^{-n}\sum_{x\in \mathbb{F}_q^n}f(x)\chi(-m\cdot x),
\end{equation}
where $\chi : \mathbb{F}_q \to S^1 \subseteq \mathbb{C}$ is a nontrivial additive character. The specific choice of $\chi$ does not play an important role in what follows. Here, $m \cdot x$ denotes the usual dot product and is an element of $\mathbb{F}_q$.

The following result, known as \textit{Plancherel's theorem}, will be used frequently:
\begin{equation}
\label{Plancherel's thm}
    \sum_{m\in \mathbb{F}_q^n}|\widehat{f}(m)|^2=q^{-n}\sum_{x\in \mathbb{F}_q^n}|f(x)|^2.
\end{equation}
For $E\subseteq \mathbb{F}_q^n$, we define $E(x)$ to be the indicator function of $E$, that is, 
\begin{equation*}
E(x) =
\begin{cases}
0, & \text{if }x\notin E, \\
1, & \text{if }x\in E.
\end{cases} 
\end{equation*}
In particular, from \eqref{Plancherel's thm}, for any set $E\subseteq \mathbb{F}_q^n$, we have:
\begin{equation}
\label{Plancherel's theorem for E}
    \sum_{m\in \mathbb{F}_q^n}|\widehat{E}(m)|^2=q^{-n}|E|=\widehat{E}(0).
\end{equation}

\subsection{Gaussian binomial coefficients}

In this subsection, we introduce the combinatorial object known as the Gaussian binomial coefficient or $q$-binomial coefficient.

\begin{defn}
\label{Gaussian binomial}
    Let $k,n\in \mathbb{N}_0$, and let $q$ be a power of a prime. The Gaussian binomial coefficient, or $q$-binomial coefficient, is defined as follows:
    \begin{equation*}
            \binom{n}{k}_q\coloneqq\begin{cases}
\dfrac{(q^n-1)(q^n-q)\dots(q^n-q^{k-1})}{(q^k-1)(q^k-q)\dots(q^k-q^{k-1})}, & \text{if } k\leq n,\\
0, & \text{if $k>n$}. 
\end{cases}
    \end{equation*}
    Note that $\binom{n}{0}_q=1$ because both the numerator and the denominator are empty products. 
\end{defn}
We also record the following definition, which was described informally above.
\begin{defn}
Let $k,n\in \mathbb{N}_0$. For $0\leq k \leq n$, let $G(k,n)$ denote the set of all $k$-dimensional linear subspaces of $\mathbb{F}_q^n$.
\end{defn} 
The following lemma demonstrates the relationship between $G(k, n)$ and $\binom{n}{k}_q$, along with other useful properties of the Gaussian binomial coefficient. These properties are analogous to those of the standard binomial coefficient, such as Pascal's rule \eqref{1st Pascal's id} and the symmetry rule \eqref{symmetry rule for GBC}. We will use these properties throughout the paper.  Although they are straightforward and well known, we could not find simple, self-contained proofs for some of them in the literature. Therefore, we provide proofs of the nontrivial parts here (see Section \ref{gaussiansection}).  

\begin{lma}
\label{relation between G(k,n) and GBC}
Let $k,n\in \mathbb{N}_0$. Then, the following statements hold: 
\begin{enumerate}
\item If $0\leq k\leq n$, then:
\begin{equation}
\label{size of G(k,n)}
|G(k,n)|=\binom{n}{k}_q.    
\end{equation}
\item Let $1\leq k\leq n$. If $z$ is a nonzero vector in $\mathbb{F}_q^n$, then:
\begin{equation}
\label{size of G(k,n) containing fixed z}
|\{V\in G(k,n): z\in V\}|=\binom{n-1}{k-1}_q.    
\end{equation}
\item Let $0\leq k\leq n$. If $z$ is a nonzero vector in $\mathbb{F}_q^n$, then:
\begin{equation}
\label{size of G(k,n):ort(V) contains z}
|\{V\in G(k,n): z\in V^{\perp}\}|=\binom{n-1}{k}_q.    
\end{equation}
\item If $q \geq 2$ and $0\leq k\leq n$, then: 
\begin{equation}
\label{ineqaulity for GBC}
q^{k(n-k)} \leq \binom{n}{k}_q \leq 4q^{k(n-k)}.  \end{equation}
\item If $1\leq k\leq n$, then:
\begin{equation}
\label{1st Pascal's id}
    \binom{n}{k}_q=\binom{n-1}{k}_q+q^{n-k}\binom{n-1}{k-1}_q,
\end{equation}
\item If $0\leq k\leq n$, then:
\begin{equation}
\label{symmetry rule for GBC}
    \binom{n}{k}_q=\binom{n}{n-k}_q.
\end{equation}
\item We have the following asymptotic relation:
\begin{equation}
\label{asymptotics of GBC}
    \binom{n}{k}_q=q^{k(n-k)}(1+O(q^{-1})) \quad \text{as} \quad q\to \infty.
\end{equation}
\end{enumerate}
\end{lma}

We defer the proof of Lemma \ref{relation between G(k,n) and GBC} to section \ref{gaussiansection}.

\section{Main results:  exceptional projections in finite fields}

Before formulating the main result of the paper, we first establish some definitions. Following \cite{fraserfinite}, we provide a framework to capture the Fourier analytic behaviour of a set $E \subseteq \mathbb{F}_q^n$.

\begin{defn}
For $E \subseteq \mathbb{F}_q^n$ and $p\in [1,+\infty)$, we define the $p$-norm of its Fourier transform as follows:
\begin{equation}
\label{p-norm of FT}
    \lVert \widehat{E}\rVert_p\coloneqq \bigg(q^{-n}\sum_{m\in \mathbb{F}_q^n\setminus \{0\}}|\widehat{E}(m)|^p \bigg)^{\frac{1}{p}}.
\end{equation}
\end{defn}

\begin{defn}
For $E \subseteq \mathbb{F}_q^n$, $p\in [1,+\infty)$, and $s\in [0,1]$, we say that $E$ is a $(p,s)$-Salem set if
\begin{equation}
\label{(p,s) Salem set}
\lVert \widehat{E}\rVert_p\lesssim_{p,s} q^{-n}|E|^{1-s}.
\end{equation}
\end{defn}

We notice that every set $E \subseteq \mathbb{F}_q^n$ is a $(2,\frac{1}{2})$-Salem set, which follows easily from \eqref{Plancherel's theorem for E}. Next, we formally define a \textit{projection} in \( \mathbb{F}_q^n \) as follows.

\begin{defn} \label{projdef}
    Let $V$ be a subspace of $\mathbb{F}_q^n$ and $E\subseteq \mathbb{F}_q^n$. The projection of $E$ onto $V$ is defined as
    \begin{equation*}
        \pi_V(E)\coloneqq \{x+V^\perp: x\in \mathbb{F}_q^n, (x+V^\perp)\cap E\neq \varnothing\}.
    \end{equation*}
\end{defn}

If $V\in G(k,n)$, it is not difficult to verify that 
\begin{equation*}
    |\pi_V(E)|\leq \min \{q^k, |E|\}.
\end{equation*}
We are interested in estimating the cardinality of the \emph{exceptional set}, which is defined as follows:
\begin{equation*}
    \{V\in G(k,n): |\pi_V(E)|\leq u\}
\end{equation*}
for $u>0$. The case of interest is $u < \min \{q^k, |E|\}$, because otherwise, the size of the exceptional set is simply $|G(k,n)|$, which, as stated in \eqref{ineqaulity for GBC}, is $\approx q^{k(n-k)}$. Now, we are ready to formulate the main result of the paper.
\begin{thm}
\label{main projection theorem}    
Let $p\in [2,+\infty), s\in [0,1]$, and $E\subseteq \mathbb{F}_q^n$ be a nonempty $(p,s)$-Salem set. If $0<u\leq \frac{1}{4} q^{\frac{2k}{p}}$, then 
\begin{equation*}
    |\{V\in G(k,n): |\pi_V(E)|\leq u\}|\lesssim_{p,s} u^{\frac{p}{2}}q^{k(n-k)}|E|^{-ps}.
\end{equation*}
\end{thm}

This theorem can be interpreted as a finite field analogue of \cite[Corollary 6.1]{ana}, which considers the continuous setting.  It is perhaps noteworthy that in \cite{ana}, there is no analogue of the condition $0<u\leq \frac{1}{4} q^{\frac{2k}{p}}$, which constrains the pairs $(u,p)$ that can be considered simultaneously. This condition appears in Theorem \ref{main projection theorem} because $0 \in \mathbb{F}_q^n$ plays a significant role  in the analysis, whereas it can be ignored completely in  the continuous case.

The fact that the upper bound in Theorem \ref{main projection theorem} depends on $p$ allows one to optimize it by choosing the $p$ from the allowed range (that is, such that $u \leq\frac{1}{4} q^{\frac{2k}{p}}$), which gives the strongest bound.  This affords the theorem considerable flexibility.  It is also worth noting that for $u \leq \frac{q^k}{4}$, we are always allowed to choose $p=2$ and for $u \lesssim   q^{k-\varepsilon}$ there is a non-trivial range of allowable $p$.  Further, by Plancherel, we may apply it in the case $p=2$ for arbitrary sets without making any further assumptions about the Fourier analytic behaviour. Indeed, by setting $p=2$ (in which case we may also set $s=\frac{1}{2}$) in Theorem \ref{main projection theorem} and following Chen's argument, we obtain \eqref{marstrandfinite} in full generality -- Marstrand's projection theorem in finite fields with Peres--Schlag--Mattila bounds on the exceptional set. This generalizes a  result of Chen \cite{chen} to the case of non-prime fields and recovers a recent result by Bright and Gan.  In fact, Bright and Gan prove something stronger but obtain \eqref{marstrandfinite} as a special case \cite[Theorem 1.6 and (13)]{bright}.

\begin{cor}
\label{main cor}    
Let $E\subseteq \mathbb{F}_q^n$ be a nonempty set. If $0<u\leq \frac{1}{4} \min\{q^{k},|E|\}$, then 
\begin{equation*}
    |\{V\in G(k,n): |\pi_V(E)|\leq u\}|\lesssim \frac{u q^{k(n-k)}}{\max\{|E|,q^k\}}.
\end{equation*}
In particular, if $0<u \leq  \frac{1}{4} \min\{q^{k},|E|\}$ and $u = o(\max\{|E|,q^k\})$, then
\[
 |\{V\in G(k,n): |\pi_V(E)|\leq u\}|=o(|G(k,n)|).
\]
\end{cor}
\begin{proof}[Proof of Corollary \ref{main cor}]
By setting $p=2$ and $s=\frac{1}{2}$ in Theorem \ref{main projection theorem}, we obtain
\begin{equation*}
    |\{V\in G(k,n): |\pi_V(E)|\leq u\}|\lesssim  u q^{k(n-k)}|E|^{-1}
\end{equation*}
which is the finite fields analogue of the exceptional set estimate due to Peres--Schlag \cite{peres}. The other estimate, which is the analogue of a result of Mattila, 
\begin{equation}
\label{Chen's bound}
    |\{V\in G(k,n): |\pi_V(E)|\leq u\}|\lesssim u q^{k(n-k)}q^{-k},
\end{equation}
was established by Chen \cite[Theorem 1.2 (a)]{chen} (Chen actually proves that the upper bound in \eqref{Chen's bound} is $u q^{k(n-k)}q^{k-n}$ because he uses a different definition for the projection $\pi_V(E)$, where projection onto $V$ refers to projection onto $V^{\perp}$.) Although Chen worked in the case of prime fields, his proof of this estimate also applies in the non-prime case.
\end{proof}

One of the main insights of this paper is that the use of $p$ in Theorem \ref{main projection theorem} provides a strengthening of \eqref{marstrandfinite}, at least in cases where good $L^p$ bounds hold for the Fourier transform of $E$.  More precisely, suppose $|E| \approx q^\alpha$ for some $\alpha \in (0,n)$ and $u \lesssim q^\beta$ for some $\beta <\min\{k,\alpha\}$. Then  Theorem \ref{main projection theorem} gives an asymptotically stronger estimate than  \eqref{marstrandfinite} whenever  the right hand side of \eqref{marstrandfinite} is a positive power of $q$ and $E$ is $(p,s)$-Salem for some $2<p<\frac{2k}{\beta}$ with
$$
s>\frac{\beta(\frac{p}{2}-1)+\max\{\alpha,k\}}{p\alpha} =\begin{cases}
\frac{\beta}{2\alpha}(1-\frac{2}{p})+\frac{1}{p}, & \text{if } \alpha \geq k, \\
\frac{\beta}{2\alpha}(1-\frac{2}{p})+\frac{k}{p\alpha}, & \text{if } \alpha < k.
\end{cases}$$

The case when $\alpha \geq k$ is particularly favorable because it was proved in \cite{fraserfinite} that \emph{all} $E$ are $(p,\frac{1}{p})$-Salem for \emph{all} $p \in [2,\infty]$. Therefore, any improvement over the trivial bound $s \geq \frac{1}{p}$ will yield an improvement over \eqref{marstrandfinite} for sufficiently small $\beta$ provided $k \leq \alpha <k(n-k)$.

Using our Fourier analytic approach, we can also obtain uniform lower bounds for the size of projections, that is, bounds that hold for \emph{all} $V \in G(k,n)$.  Another way to view these results is that the exceptional set is empty for small enough $u$.

\begin{thm}
\label{main projection theorem2}    
Let $p\in [2,+\infty), s\in [0,1]$, and let $E\subseteq \mathbb{F}_q^n$ be a nonempty $(p,s)$-Salem set. Then, for all $V \in G(k,n)$,
\[
|\pi_V(E)|^{\frac{1}{2}} \geq \frac{1}{C  q^{\frac{n-k}{p}} |E|^{ -s} + q^{-\frac{k}{2}}  } 
\]
where $C\approx_{p,s} 1$ is the implicit constant from the definition of a $(p,s)$-Salem set. In particular, for all $V \in G(k,n)$,
\[
|\pi_V(E)|  \gtrsim \min\left\{\frac{|E|^{2s}}{q^{\frac{2(n-k)}{p}}}, q^{k}\right\}.
\]
\end{thm}

The case $p\to\infty$ in Theorem \ref{main projection theorem2} recovers and extends \cite[Proposition 1.7]{chen}, which only applies in the case of prime fields. However, Theorem \ref{main projection theorem2} may provide an even stronger estimate than this because there is no reason for the bound to be optimized as $p \to \infty$.

\section{Character sums for subspaces and a problem of Chen}

We first establish a simple identity for character sums associated with subspaces, which is interesting in its own right. It generalizes Lemma 2.4 from \cite{chen}, which was proven only for the case when $q$ is a prime. In \cite{chen}, the author emphasizes that he does not know whether the result holds for vector spaces over general finite fields, see \cite[Remark 2.5]{chen}. Bright and Gan   hinted that they believe  Chen's result holds in the non-prime case, but did not provide a proof \cite[Remark 5]{bright}. We provide a direct and elementary argument.

In the following, if $U$ is a subspace of $\mathbb{F}_q^n$, then $U^{\perp}$ denotes the orthogonal complement of $U$. It is a well-known fact that if $\dim(U)=k$, then $\dim(U^{\perp})=n-k$, and hence $|U^{\perp}|=q^{n-k}$.

\begin{lma}
\label{sum of characters}
    Let $k,n\in \mathbb{N}_0$ such that $0\leq k\leq n$, and let $U\in G(n-k,n)$. If $x\notin U$, then 
    \begin{equation*}
        \sum_{y\in U^{\perp}}\chi(-x\cdot y)=0,
    \end{equation*}
    where $\chi:\mathbb{F}_q\to S^1\subseteq \mathbb{C}$ is a non-trivial additive character.
\end{lma}

\begin{proof}[Proof of Lemma \ref{sum of characters}]    The case $k=0$ is vacuously true. Thus, we assume that $k>0$. Let $q$ be a power of a prime $p$. (The use of $p$ in this lemma should not be confused with the use of $p$ throughout the rest of the paper in the context of $L^p$ norms of the Fourier transform.) Since $\chi:\mathbb{F}_q\to S^1\subseteq \mathbb{C}$ is a non-trivial additive character, it follows that $\chi(\mathbb{F}_q)=\mu_p$, where $\mu_p$ is the cyclic group of $p$-th roots of unity, i.e., $\mu_p=\{z\in \mathbb{C}: z^p=1\}$. 
For $s\in \chi(\mathbb{F}_q)$, we define
    \begin{equation*}
        A(s)\coloneqq \{y\in U^{\perp}: \chi(-y\cdot x)=s\}.
    \end{equation*}
    Notice that the set $\{A(s): s\in \chi(\mathbb{F}_q)\}$ forms a partition of $U^{\perp}$, i.e.,
    \begin{equation}
    \label{partition of V^{perp}}
        \bigsqcup\limits_{s\in \chi(\mathbb{F}_q)}A(s)=U^{\perp}.
    \end{equation}
    Observe that
    \begin{equation*}
        U=\{v\in \mathbb{F}_q^n: \forall y\in U^{\perp} \ (v\cdot y=0) \}.
    \end{equation*}
    Since $x\notin U$, there exists $y_0\in U^{\perp}$ such that $x\cdot y_0\neq 0$. Since $\chi$ is a nontrivial character, it follows that $\chi(\alpha)\neq 1$ for some $\alpha\in \mathbb{F}_q$. Consider 
    \begin{equation*}
        y\coloneqq \underbrace{\alpha(x\cdot y_0)^{-1}}_{\in \mathbb{F}_q}y_0\in U^{\perp}.    
    \end{equation*}
    Observe that 
    \begin{equation*}
    y\cdot x=(\alpha(x\cdot y_0)^{-1}y_0)\cdot x=\alpha(x\cdot y_0)^{-1}(y_0\cdot x)=\alpha.    
    \end{equation*}
         Therefore, $\chi(y\cdot x)=\chi(\alpha)\neq 1$, which implies that $\chi(-y\cdot x)\neq 1$. We have shown that there exists $y\in U^{\perp}$ such that $\chi(-y\cdot x)\neq 1$, which implies that $A(s)\neq \varnothing$, where $s\coloneqq \chi(-y\cdot x)\neq 1$. Fix $y\in A(s)$ and consider the mapping 
    \begin{equation*}
    f:U^{\perp}\to U^{\perp}, \quad  v\xmapsto{f} v+y,    
    \end{equation*}
     which is well-defined and injective.     We claim that $f(A(s))\subseteq A(s^2)$. Indeed, let $u\in f(A(s))$, so $u=f(v)$ for some $v\in A(s)$. Notice that $u\in U^{\perp}$ and 
    \begin{equation*}
        \chi(-u\cdot x)=\chi(-f(v)\cdot x)=\chi(-(v+y)\cdot x)=\chi(-v\cdot x)\chi(-y\cdot x)=s^2,
    \end{equation*}
    implying $u\in A(s^2)$. Therefore, the mapping $f:A(s)\to A(s^2)$ is well-defined, and $|A(s)|\leq |A(s^2)|$ because $f$ is injective. More generally, one can show that $|A(s^k)|\leq |A(s^{k+1})|$ for every positive integer $k$. Since $s\in \mu_p\setminus \{1\}$, its order is $p$, implying $s^{p+1}=s$. Therefore, we have:
\begin{equation*}
    |A(s)|\leq |A(s^2)|\leq \cdots \leq |A(s^p)|\leq |A(s^{p+1})|=|A(s)|, 
\end{equation*}
which implies 
\begin{equation*}
    |A(s)|=|A(s^2)|=\dots=|A(s^p)|,
\end{equation*}
where all the elements $\{s,s^2,\dots,s^p\}$ are distinct because $s\neq 1$. This immediately implies that $|A(s_1)|=|A(s_2)|$ for every $s_1,s_2\in \chi(\mathbb{F}_q)$. For simplicity, assume that $k$ is their common value, i.e., $k\coloneqq|A(s)|$ for every $s\in \chi(\mathbb{F}_q)$. Then, by \eqref{partition of V^{perp}}, 
    \begin{align*}
        \sum_{y\in U^{\perp}}\chi(-x\cdot y)&=\sum_{s\in \chi(\mathbb{F}_q)}\sum_{y\in A(s)}\chi(-x\cdot y)=\sum_{s\in \chi(\mathbb{F}_q)}\sum_{y\in A(s)}s\\
        &=\sum_{s\in \chi(\mathbb{F}_q)}s|A(s)|=k\sum_{s\in \chi(\mathbb{F}_q)}s=0. \qedhere
    \end{align*}
\end{proof}

\subsection{Plancherel's theorem for subspaces}

In this subsection, we establish  a generalization of Plancherel's theorem \eqref{Plancherel's thm} for subspaces of $\mathbb{F}_q^n$; see Theorem \ref{complete Plancherel on subspaces}. This result fully generalizes \cite[Lemma 2.3]{chen}, which considered only the case of prime fields and indicator functions of sets, although we follow Chen's argument. Note that Theorem \ref{complete Plancherel on subspaces} indeed generalizes \eqref{Plancherel's thm} when we take $U=\{0\}$. Before stating and proving the theorem, we first introduce some preliminaries and notation.

Let $U\in G(n-k,n)$, so $|U|=q^{n-k}$. Hence, the total number of distinct cosets of $U$ in $\mathbb{F}_q^n$ is  
\begin{equation*}
    [\mathbb{F}_q^n:U]=\frac{|\mathbb{F}_q^n|}{|U|}=q^k.    
\end{equation*}

We write  the distinct cosets of $U$ in $\mathbb{F}_q^n$ as $\{U+t_{U,j}: 1\leq j \leq q^k\}$, and note  that they form a partition of $\mathbb{F}_q^n$, i.e.,
\begin{equation}
\label{partition of space}
    \mathbb{F}_q^n=\bigsqcup_{j=1}^{q^k}(U+t_{U,j}).
\end{equation}

Moreover, if $f:\mathbb{F}_q^n\to \mathbb{C}$ and $A\subseteq \mathbb{F}_q^n$, we define
\begin{equation*}
    S_f(A)\coloneqq \sum_{a\in A}f(a).
\end{equation*}

\begin{thm}[Plancherel's theorem for subspaces]
\label{complete Plancherel on subspaces}
    Let $E\subseteq \mathbb{F}_q^n$ and $U\in G(n-k,n)$. Then for any $f:\mathbb{F}_q^n\to \mathbb{C}$ we have:
    \begin{equation}
    \label{complete generalization of Plancherel}
        \sum_{m\in U^{\perp}}|\widehat{f}(m)|^2=q^{-2n+k}\sum_{j=1}^{q^k}|S_f(U+t_{U,j})|^2.
    \end{equation}
\end{thm}

\begin{proof}[Proof of Theorem \ref{complete Plancherel on subspaces}]

Let $m\in U^{\perp}$. Then, by \eqref{FT} and \eqref{partition of space}, 
\begin{align*}
    \widehat{f}(m)=q^{-n}\sum_{x\in \mathbb{F}_q^n}f(x)\chi(-x\cdot m) 
    &=q^{-n}\sum_{j=1}^{q^k}\sum_{x\in U+t_{U,j}}f(x)\chi(-x\cdot m) \\
    &=q^{-n}\sum_{j=1}^{q^k}\sum_{u\in U}f(u+t_{U,j})\chi(-(u+t_{U,j})\cdot m) \\
    &=q^{-n}\sum_{j=1}^{q^k}\sum_{u\in U}f(u+t_{U,j})\chi(-t_{U,j}\cdot m) \\
    &=q^{-n}\sum_{j=1}^{q^k}\chi(-m\cdot t_{U,j})S_f(U+t_{U,j}). 
\end{align*}
Given this expression for $\widehat{f}(m)$, we can compute $|\widehat{f}(m)|^2$ as follows:
\begin{align*}
    |\widehat{f}(m)|^2&=q^{-2n}\sum_{i,j=1}^{q^k}\chi(m\cdot (t_{U,i}- t_{U,j}))S_f(U+t_{U,j}) \overline{S_f(U+t_{U,i})} \nonumber \\
    &=q^{-2n}\sum_{j=1}^{q^k}|S_f(U+t_{U,j})|^2  +q^{-2n}\sum_{\substack{i,j=1 \\ i\neq j}}^{q^k}\chi(m\cdot (t_{U,i}- t_{U,j}))S_f(U+t_{U,j}) \overline{S_f(U+t_{U,i})}.
\end{align*}
Summing over all $m\in U^{\perp}$ and taking into account that $|U^{\perp}|=q^k$ since $U\in G(n-k,n)$, we have:
\begin{align}
    \sum_{m\in U^{\perp}}|\widehat{f}(m)|^2&=q^{-2n+k}\sum_{j=1}^{q^k}|S_f(U+t_{U,j})|^2 \nonumber \\
    &+q^{-2n}\sum_{\substack{i,j=1 \\ i\neq j}}^{q^k}S_f(U+t_{U,j}) \overline{S_f(U+t_{U,i})}\sum_{m\in U^{\perp}}\chi(m\cdot (t_{U,i}- t_{U,j})) \label{inner zero sum} .
\end{align}
We observe that since $U\in G(n-k,n)$ and $i\neq j$, it follows that $t_{U,i}-t_{U,j}\notin U$. Indeed, if $t_{U,i}-t_{U,j}\in U$, this would imply that $t_{U,i}+U=t_{U,j}+U$, which is a contradiction since the cosets $t_{U,i}+U$ and $t_{U,j}+U$ are distinct because $i\neq j$. Hence, the inner sum in \eqref{inner zero sum} vanishes by Lemma \ref{sum of characters}, completing the proof.
\end{proof}

The following result will be very useful later and follows immediately from Theorem \ref{complete Plancherel on subspaces} by taking $f:\mathbb{F}_q^n\to \mathbb{C}$ to be the indicator function of a set $E\subseteq \mathbb{F}_q^n$, i.e., $f(x)=E(x)$.

\begin{cor}
\label{Plancherel on subspaces}
    Let $E\subseteq \mathbb{F}_q^n$ and $U\in G(n-k,n)$. Then we have:
    \begin{equation}
    \label{generalization of Plancherel}
        \sum_{m\in U^{\perp}}|\widehat{E}(m)|^2=q^{-2n+k}\sum_{j=1}^{q^k}|E\cap (U+t_{U,j})|^2.
    \end{equation}
\end{cor}

\begin{proof}[Proof of Corollary \ref{Plancherel on subspaces}] The proof follows directly from the following simple observation: if we let $f:\mathbb{F}_q^n\to \mathbb{C}$ be defined by $f(x)=E(x)$, then
\begin{equation*}
    S_f(U+t_{U,j})=\sum_{x\in U+t_{U,j}}E(x)=|E\cap (U+t_{U,j})|. \qedhere
\end{equation*}
\end{proof}

\section{Gaussian binomial coefficients: proof of Lemma \ref{relation between G(k,n) and GBC}} \label{gaussiansection}

\begin{enumerate}
    \item The case $k=0$ is trivial since both $|G(0,n)|$ and $\binom{n}{0}_q$ are equal to $1$. Thus, we assume that $k>0$ and define 
    \begin{equation*}
        \mathcal{A}\coloneqq\{(u_1,\dots,u_k)\in (\mathbb{F}_q^n)^k: \{u_1,\dots,u_k\} \ \textup{is linearly independent}\},
    \end{equation*}
    along with the mapping 
    \begin{equation*}
    f: \mathcal{A}\to G(k,n), \quad (u_1,\dots,u_k)\xmapsto{f} \textup{span}(u_1,\dots,u_k).   
    \end{equation*}
    It is straightforward to verify that the mapping $f$ is well-defined and surjective. Moreover, it is not difficult to show that for every $V\in G(k,n)$, we have: 
    \begin{equation*}
        |f^{-1}(\{V\})|=(q^k-1)(q^k-q)\dots(q^k-q^{k-1}).
    \end{equation*}
        Since $\mathcal{A}=f^{-1}(G(k,n))$, we obtain the following relation:
    \begin{equation}
    \label{relation between A and G(k,n)}
        |\mathcal{A}|=|G(k,n)|\cdot (q^k-1)(q^k-q)\dots(q^k-q^{k-1}).
    \end{equation}
    Similarly, one can compute that 
    \begin{equation}
    \label{value of |A|}
        |\mathcal{A}|=(q^n-1)(q^n-q)\dots (q^n-q^{k-1}).
    \end{equation}
    Substituting the value of $|\mathcal{A}|$ from \eqref{value of |A|} into \eqref{relation between A and G(k,n)}, we obtain:
    \begin{equation*}
        |G(k,n)|=\frac{(q^n-1)(q^n-q)\dots (q^n-q^{k-1})}{(q^k-1)(q^k-q)\dots(q^k-q^{k-1})},
    \end{equation*}
    or equivalently, $|G(k,n)|=\binom{n}{k}_q$, which completes the proof of \eqref{size of G(k,n)}.

\medskip

    \item The proof of this part can be carried out in a similar manner to part (1), so we omit it.

\medskip

    \item The case $k=0$ is trivial, and thus we assume that $k>0$. Let $z\in \mathbb{F}_q^n$ be a nonzero vector. It is then clear that
    \begin{equation}
    \label{equivalence}
        \{V\in G(k,n): z\in V^{\perp}\}=\{V\in G(k,n): V\subseteq (\text{span}(z))^{\perp}\}.
    \end{equation}
        We observe that $(\text{span}(z))^{\perp}$ is a $(n-1)$-dimensional subspace of $\mathbb{F}_q^n$, and hence $(\text{span}(z))^{\perp}\cong \mathbb{F}_q^{n-1}$. It is easy to demonstrate that 
    \begin{equation}
    \label{in terms of G(k,n-1)}
        |\{V\in G(k,n): V\subseteq (\text{span}(z))^{\perp}\}|=|G(k,n-1)|.
    \end{equation}
    By combining \eqref{equivalence} and \eqref{in terms of G(k,n-1)}, we obtain        $|\{V\in G(k,n): z\in V^{\perp}\}|=|G(k,n-1)|$,
    and, by \eqref{size of G(k,n)}, we have 
    \begin{equation*}
    |\{V\in G(k,n): z\in V^{\perp}\}|=\binom{n-1}{k}_q,    
    \end{equation*}
    which finishes the proof of \eqref{size of G(k,n):ort(V) contains z}.
    
\medskip

    \item The case $k=0$ is trivial since the inequality we wish to prove reduces to the trivial inequality $1\leq 1\leq 4$. Hence, we will assume that $k>0$. The lower bound can be obtained easily as follows:
    \begin{equation*}
        \binom{n}{k}_q=\prod_{i=0}^{k-1}\frac{q^n-q^i}{q^k-q^i}\geq \prod_{i=0}^{k-1}q^{n-k}=q^{k(n-k)}.
    \end{equation*}
        The upper bound requires a little more work. Trivially, we have
    \begin{equation}
    \label{upper bound for numerator}
        \prod_{i=0}^{k-1}(q^n-q^i)\leq q^{nk}.
    \end{equation}
        On the other hand, since $q\geq 2$, we have:
    \begin{equation} \label{lower bound for denom}
        \prod_{i=0}^{k-1}(q^k-q^i)=q^{k^2}\prod_{i=0}^{k-1}(1-q^{i-k})=q^{k^2}\prod_{i=1}^{k}\left(1-\frac{1}{q^i}\right)\geq q^{k^2}\prod_{i=1}^{k}\left(1-\frac{1}{2^i}\right) \geq \frac{q^{k^2}}{4}.
    \end{equation}
Combining the bounds \eqref{lower bound for denom} and \eqref{upper bound for numerator}, we obtain:
\begin{equation*}
    \binom{n}{k}_q=\frac{\prod\limits_{i=0}^{k-1}(q^n-q^i)}{\prod\limits_{i=0}^{k-1}(q^k-q^i)}\leq 4q^{k(n-k)},
\end{equation*}
which completes the proof of \eqref{ineqaulity for GBC}. 

\medskip

\item We omit the proof of this part, as it involves only straightforward algebraic computations.

\smallskip

\item We assume that $k>0$, as the case $k=0$ is trivial. First, assume that $k\leq n-k$, i.e., $2k\leq n$. Then, by the definition of the Gaussian binomial coefficient, we have:
\begin{align*}
    \binom{n}{k}_q&=\prod_{i=0}^{k-1}\frac{q^n-q^i}{q^k-q^i}=q^{k(n-2k)}\prod_{i=0}^{k-1}\frac{q^n-q^i}{q^{n-k}-q^{n-2k+i}}\\
    &=q^{k(n-2k)}\prod_{i=0}^{k-1}\frac{q^n-q^i}{q^{n-k}-q^{n-2k+i}} \prod_{i=0}^{n-2k-1}\frac{q^{n-k}-q^i}{q^{n-k}-q^i}\\
    &=\prod_{i=0}^{k-1}\frac{q^n-q^i}{q^{n-k}-q^{n-2k+i}} \prod_{i=0}^{n-2k-1}\frac{q^{n}-q^{i+k}}{q^{n-k}-q^i}\\
    &=\prod_{i=0}^{n-k-1}\frac{q^n-q^i}{q^{n-k}-q^i}=\binom{n}{n-k}_q.
\end{align*}
The case $k>n-k$ follows from the previous one by setting $m\coloneqq n-k$.
\end{enumerate}

\section{Exceptional projections: proof of Theorem \ref{main projection theorem}} \label{proofmain}

Let $E\subseteq \mathbb{F}_q^n$ and $u>0$. Consider the exceptional set defined as follows: 
\begin{equation*}
    \Theta \coloneqq \{V\in G(k,n): |\pi_V(E)|\leq u\}.    
\end{equation*}
We assume that $|\Theta|>0$; otherwise, there is nothing to prove. For $V\in G(k,n)$, we know that $V^{\perp}\in G(n-k,n)$, and the total number of distinct cosets of $V^{\perp}$ in $\mathbb{F}_q^n$ is $q^k$. Suppose the distinct cosets of $V^{\perp}$ in $\mathbb{F}_q^n$ are given by $\{V^{\perp}+t_{V^{\perp},j}: 1\leq j \leq q^k\}$, noting that they form a partition of $\mathbb{F}_q^n$. Therefore, we have:
\begin{equation*}
    E=\bigsqcup\limits_{j=1}^{q^k}E\cap (V^{\perp}+t_{V^{\perp},j}),
\end{equation*}
implying
\begin{equation*}
    |E|=\sum_{j=1}^{q^k}|E\cap (V^{\perp}+t_{V^{\perp},j})|.
\end{equation*} 
Applying the Cauchy–Schwarz inequality and using Definition \ref{projdef}, we obtain:
\begin{equation}
\label{1st inequality}
    |E|\leq |\pi_V(E)|^{\frac{1}{2}} \Bigg(\sum_{j=1}^{q^k}|E\cap (V^{\perp}+t_{V^{\perp},j})|^2\Bigg)^{\frac{1}{2}}.
\end{equation}
Summing $\eqref{1st inequality}$ over all $V\in \Theta$, we obtain:
\begin{equation*}
    |\Theta||E|\leq u^{\frac{1}{2}}\sum_{V\in \Theta}\Bigg(\sum_{j=1}^{q^k}|E\cap (V^{\perp}+t_{V^{\perp},j})|^2\Bigg)^{\frac{1}{2}}.
\end{equation*}
Applying the Cauchy–Schwarz inequality once more, we have:
\begin{equation*}
    |\Theta||E|\leq u^{\frac{1}{2}}|\Theta|^{\frac{1}{2}}\Bigg(\sum_{V\in \Theta}\sum_{j=1}^{q^k}|E\cap (V^{\perp}+t_{V^{\perp},j})|^2\Bigg)^{\frac{1}{2}},
\end{equation*}
which leads to 
\begin{equation}
\label{2nd inequality}
|E|\leq u^{\frac{1}{2}} |\Theta|^{-\frac{1}{2}} \Bigg(\sum_{V\in \Theta}\sum_{j=1}^{q^k}|E\cap (V^{\perp}+t_{V^{\perp},j})|^2\Bigg)^\frac{1}{2}.
\end{equation}
Applying Corollary \ref{Plancherel on subspaces} to the right-hand side of \eqref{2nd inequality} with $U=V^{\perp}$ and noting that $U^{\perp}=V$, we obtain:
\begin{align*}
|E|&\leq u^{\frac{1}{2}}q^n \Bigg(|\Theta|^{-1}\sum_{V\in \Theta}q^{-k}\sum_{m\in V}|\widehat{E}(m)|^2\Bigg)^\frac{1}{2}.
\end{align*}
Since $p\in [2,+\infty)$, we can apply Hölder's inequality to derive:
\begin{align}
|E|&\leq u^{\frac{1}{2}}q^n \Bigg(|\Theta|^{-1}\sum_{V\in \Theta}q^{-k}\sum_{m\in V}|\widehat{E}(m)|^p\Bigg)^\frac{1}{p} \nonumber \\
&=u^{\frac{1}{2}}q^{n-\frac{k}{p}} |\Theta|^{-\frac{1}{p}} \Bigg(\sum_{V\in \Theta}\sum_{m\in V}|\widehat{E}(m)|^p\Bigg)^\frac{1}{p} \label{inner sum}.    
\end{align}
We observe that $0\in V$ for every $V\in \Theta$, and we will split the inner sum in \eqref{inner sum} as follows:
\begin{align}
    |E|&\leq u^{\frac{1}{2}}q^{n-\frac{k}{p}} |\Theta|^{-\frac{1}{p}} \Bigg(\sum_{V\in \Theta}\sum_{m\in V\setminus \{0\}}|\widehat{E}(m)|^p+\sum_{V\in \Theta}|\widehat{E}(0)|^p\Bigg)^\frac{1}{p} \nonumber \\
    &=u^{\frac{1}{2}}q^{n-\frac{k}{p}} |\Theta|^{-\frac{1}{p}} \Bigg(\sum_{V\in \Theta}\sum_{m\in V\setminus \{0\}}|\widehat{E}(m)|^p+|\Theta||\widehat{E}(0)|^p\Bigg)^\frac{1}{p} \label{sum 1} \\
    &\leq u^{\frac{1}{2}}q^{n-\frac{k}{p}} |\Theta|^{-\frac{1}{p}} \Bigg[\bigg(\sum_{V\in \Theta}\sum_{m\in V\setminus \{0\}}|\widehat{E}(m)|^p\bigg)^{\frac{1}{p}}+\bigg(|\Theta||\widehat{E}(0)|^p\bigg)^\frac{1}{p}\Bigg] \label{sum 2} \\
    &=u^{\frac{1}{2}}q^{n-\frac{k}{p}} |\Theta|^{-\frac{1}{p}} \Bigg(\sum_{V\in \Theta}\sum_{m\in V\setminus \{0\}}|\widehat{E}(m)|^p\Bigg)^\frac{1}{p}+u^{\frac{1}{2}}q^{-\frac{k}{p}}|E| \label{nonzero double sum} .
\end{align}
We note that \eqref{sum 2} was derived from \eqref{sum 1} using the following trivial inequality: if $a,b\in [0,+\infty)$ and $p\in [1,+\infty)$, then $(a+b)^{\frac{1}{p}}\leq a^{\frac{1}{p}}+b^{\frac{1}{p}}$.
We proceed by estimating the double sum in \eqref{nonzero double sum} as follows.  Since $\Theta \subseteq G(k,n)$, we have:
\begin{align*}
    \sum_{V\in \Theta}\sum_{m\in V\setminus \{0\}}|\widehat{E}(m)|^p &\leq \sum_{V\in G(k,n)}\sum_{m\in V\setminus \{0\}}|\widehat{E}(m)|^p\\
    &=\sum_{m\in \mathbb{F}_q^n\setminus \{0\}}\sum_{\substack{V\in G(k,n)\\ m\in V}}|\widehat{E}(m)|^p\\
    &=\sum_{m\in \mathbb{F}_q^n\setminus \{0\}} |\widehat{E}(m)|^p \cdot |\{V\in G(k,n):m\in V\}|.
\end{align*}
By consecutively applying \eqref{size of G(k,n) containing fixed z} and \eqref{ineqaulity for GBC}, we obtain the following upper bound:
\begin{align}
    \sum_{V\in \Theta}\sum_{m\in V\setminus \{0\}}|\widehat{E}(m)|^p & \leq \binom{n-1}{k-1}_q \ \sum_{m\in \mathbb{F}_q^n\setminus \{0\}} |\widehat{E}(m)|^p \nonumber \\
    &\leq 4q^{(k-1)(n-k)}\sum_{m\in \mathbb{F}_q^n\setminus \{0\}} |\widehat{E}(m)|^p \nonumber \\
    &=4q^{(k-1)(n-k)+n} \lVert \widehat{E}\rVert_p^p \nonumber \\
    &=4q^{k(n-k)+k}\lVert \widehat{E}\rVert_p^p \label{p-norm}.
\end{align}
By combining the estimates \eqref{p-norm} and \eqref{nonzero double sum}, we derive the following upper bound for $|E|$:
\begin{align}
\label{upper bound for E}
    |E|\leq 2u^{\frac{1}{2}} q^{n+\frac{k(n-k)}{p}} |\Theta|^{-\frac{1}{p}}  \lVert \widehat{E}\rVert_p +u^{\frac{1}{2}}q^{-\frac{k}{p}}|E|.
\end{align}
Since we assumed that $E$ is a $(p,s)$-Salem set, we have $\lVert \widehat{E}\rVert_p \leq Cq^{-n} |E|^{1-s}$ for some $C=C(p,s)$, and by substituting this into \eqref{upper bound for E}, we obtain:
\begin{align}
\label{upper bound for E (2)}
    |E|\leq 2Cu^{\frac{1}{2}} q^{\frac{k(n-k)}{p}} |\Theta|^{-\frac{1}{p}} |E|^{1-s} +u^{\frac{1}{2}}q^{-\frac{k}{p}}|E|.
\end{align}
Given that $0<u\leq \frac{1}{4}q^{\frac{2k}{p}}$, performing some arithmetic manipulations yields the following upper bound:
\begin{align*}
    |\Theta|\leq (4C)^p u^{\frac{p}{2}}q^{k(n-k)}|E|^{-ps} \lesssim_{p,s} u^{\frac{p}{2}}q^{k(n-k)}|E|^{-ps},
\end{align*}
which completes the proof of Theorem \ref{main projection theorem}.

\section{Uniform lower bounds: proof of Theorem \ref{main projection theorem2}} \label{proofmain2}
Fix $V \in G(k,n)$. Following the beginning of the proof of Theorem \ref{main projection theorem}, we recall that \eqref{1st inequality} gives
\[
    |E|\leq |\pi_V(E)|^{\frac{1}{2}} \Bigg(\sum_{j=1}^{q^k}|E\cap (V^{\perp}+t_{V^{\perp},j})|^2\Bigg)^{\frac{1}{2}}.
\]
Applying Corollary \ref{Plancherel on subspaces} to the right-hand side, we obtain
\begin{align*}
 |E|&\leq |\pi_V(E)|^{\frac{1}{2}} q^n\Bigg(q^{-k} \sum_{m \in V}|\widehat{E}(m)|^2\Bigg)^{\frac{1}{2}}\\
 &\leq |\pi_V(E)|^{\frac{1}{2}} q^n\Bigg(q^{-k} \sum_{m \in V\setminus\{0\}}|\widehat{E}(m)|^2\Bigg)^{\frac{1}{2}}+|\pi_V(E)|^{\frac{1}{2}} q^{n-\frac{k}{2}} |\widehat{E}(0)|\\
 &\leq |\pi_V(E)|^{\frac{1}{2}} q^n\Bigg(q^{-k} \sum_{m \in V\setminus\{0\}}|\widehat{E}(m)|^p\Bigg)^{\frac{1}{p}}+|\pi_V(E)|^{\frac{1}{2}} q^{-\frac{k}{2}} |E| \\
  &\leq |\pi_V(E)|^{\frac{1}{2}} q^n\Bigg(q^{-k} \sum_{m \in \mathbb{F}_q^n\setminus\{0\}}|\widehat{E}(m)|^p\Bigg)^{\frac{1}{p}}+|\pi_V(E)|^{\frac{1}{2}} q^{-\frac{k}{2}} |E|  \\
    &= |\pi_V(E)|^{\frac{1}{2}} q^{n+\frac{n-k}{p}}\Bigg(q^{-n} \sum_{m \in \mathbb{F}_q^n\setminus\{0\}}|\widehat{E}(m)|^p\Bigg)^{\frac{1}{p}} +|\pi_V(E)|^{\frac{1}{2}} q^{-\frac{k}{2}} |E| \\
    &= |\pi_V(E)|^{\frac{1}{2}} q^{n+\frac{n-k}{p}}\|\widehat E \|_p +|\pi_V(E)|^{\frac{1}{2}} q^{-\frac{k}{2}} |E|.
\end{align*}
Since we assumed that $E$ is a $(p,s)$-Salem set, we have $\lVert \widehat{E}\rVert_p \leq Cq^{-n} |E|^{1-s}$ for some constant $C=C(p,s)$.  Therefore, 
\[
|E| \leq C|\pi_V(E)|^{\frac{1}{2}} q^{\frac{n-k}{p}} |E|^{1-s} +|\pi_V(E)|^{\frac{1}{2}} q^{-\frac{k}{2}} |E|.
\]
Rearranging, we obtain
\[
|\pi_V(E)|^{\frac{1}{2}} \geq \frac{1}{C  q^{\frac{n-k}{p}} |E|^{ -s} + q^{-\frac{k}{2}}  }  
\]
which also implies
\[
|\pi_V(E)|  \gtrsim \min\left\{\frac{|E|^{2s}}{q^{\frac{2(n-k)}{p}}}, q^{k}\right\},
\]
completing the proof.

\section{Another application: incidences in \texorpdfstring{$\mathbb{F}_q^n$}{Fqn} via Fourier analysis}
\subsection{Background and context}

In this section, we discuss affine incidence geometry in more detail and recall precise definitions of  the main objects.

\begin{defn}
Let $k,n\in \mathbb{N}_0$. For $0\leq k\leq n$, let $A(k,n)$ denote the set of all $k$-dimensional affine subspaces of $\mathbb{F}_q^n$, also referred to as $k$-planes.  More formally, we define
    \begin{equation*}
        A(k,n)\coloneqq \{x+V: x\in \mathbb{F}_q^n, V\in G(k,n)\}.
    \end{equation*}
\end{defn}
It is useful to note that $|A(k,n)| = q^{n-k}|G(k,n)|$ and that the elements of $A(k,n)$ form a partition of $\mathbb{F}_q^n$; see Section \ref{moment proof} ($r=0$) for a precise derivation of this.  Now, we can define the incidence function, which is the main object of study in this section.

\begin{defn}[Incidence function]
Let $k,n\in \mathbb{N}_0$, and let $E\subseteq \mathbb{F}_q^n$ and $\mathcal{A}\subseteq A(k,n)$. The incidence function between $E$ and $\mathcal{A}$ is defined as
\begin{equation}
\label{incidence function}
    I(E,\mathcal{A})\coloneqq |\{(p,A)\in E\times \mathcal{A}: p\in A\}|.
\end{equation}    
\end{defn}
In order to estimate $I(E,\mathcal{A})$, it will be useful to consider the following moment sums.  We establish explicit formulas for the 0th, 1st, and 2nd moments, which may be of independent interest (see Lemma \ref{0-th moment}).
\begin{defn}
Let $k,n\in \mathbb{N}_0$, and let $E\subseteq \mathbb{F}_q^n$ and $\mathcal{A}\subseteq A(k,n)$. For $r\in \mathbb{N}_0$, we define the $r$th moment of $E$ over $\mathcal{A}$ as
    \begin{equation}
    \label{equation moment}
        \mathcal{E}_r(E,\mathcal{A})\coloneqq \sum_{A\in \mathcal{A}}|E\cap A|^r.
    \end{equation}
\end{defn}

Before formulating the main results, we briefly discuss previous incidence results in both the continuous and finite field settings, as well as the connection between the introduced definitions.

The incidence function $I(E,\mathcal{A})$ defined in \eqref{incidence function} plays a fundamental role in incidence geometry. One of the most well-known results in this area is the Szemerédi–Trotter theorem \cite{szem}, which states that if $ E $ is a finite set of points and $\mathcal{L}$ is a finite set of lines, both in $\mathbb{R}^2$, then
\begin{equation*}
I( E ,\mathcal{L})\lesssim | E |^\frac{2}{3}|\mathcal{L}|^\frac{2}{3}+| E |+|\mathcal{L}|.    
\end{equation*}
Vinh \cite{vinh} established an analogue of the Szemerédi–Trotter theorem for points and lines in $\mathbb{F}_q^2$ using spectral graph theory. More precisely, he proved that if $ E $ is a collection of points and $\mathcal{L}$ is a collection of lines, both in $\mathbb{F}_q^2$, then
\begin{equation}
\label{vinh's result}
    I( E ,\mathcal{L})\leq \frac{| E ||\mathcal{L}|}{q}+\left(q| E ||\mathcal{L}|\right)^{\frac{1}{2}}.
\end{equation}
He also proved an analogue of \eqref{vinh's result} for point-hyperplane incidences in higher dimensions which is formulated as follows: if $ E $ is a collection of points in $\mathbb{F}_q^n$ and $\mathcal{H}$ is a collection of hyperplanes in $\mathbb{F}_q^n$ with $n\geq 2$, then
\begin{equation}
\label{vinh's higher dimensions}
    I( E ,\mathcal{H})\leq \frac{| E ||\mathcal{H}|}{q}+\left(q^{n-1}| E ||\mathcal{H}|\right)^{\frac{1}{2}}(1+o(1)).
\end{equation}
The next natural generalization of this would be to consider collections of affine $k$-planes for $1 \leq k \leq n-2$, but the literature here becomes a bit more complicated. The following result was obtained by Lund and Saraf \cite[Corollary 2]{lund} and, independently, by Bennett, Iosevich, and Pakianathan \cite[Theorem 2.3]{bennett} (in the arXiv preprint, \emph{not} the published version): if $ E $ is a collection of points in $\mathbb{F}_q^n$ and $\mathcal{A}$ is a collection of $k$-planes in $\mathbb{F}_q^n$ with $n>k \geq 1$, then
\begin{equation}
\label{vinh's higher dimensions for k}
    \left|I( E ,\mathcal{A}) -  \frac{| E ||\mathcal{A}|}{q^{n-k}} \right| \leq \left(q^{k(n-k)}| E ||\mathcal{A}|\right)^{\frac{1}{2}}(1+o(1)).
\end{equation}
Despite sustained attention on this problem over many years, it was surprisingly only recently rediscovered that \eqref{vinh's higher dimensions for k} actually goes back to the 1980s \cite{haemers}, see also \cite{dvirda}. The proofs in \cite{vinh,lund,bennett,haemers} all rely on spectral graph theory. 

The estimate \eqref{vinh's higher dimensions for k} is the current state of the art, although we note that there are some claimed improvements over this bound in the literature, such as \cite[Theorem 2.1]{bennett} (the published version). We observe below that this claimed improvement  is incorrect.\footnote{The fact that the published version of \cite[Theorem 2.1]{bennett} is incorrect is known to the authors of that paper (personal communication, A.~Iosevich), but we are not aware of this having been recorded in the literature.}

We next discuss the relationship between the moments $\mathcal{E}_r(E,\mathcal{A})$ and the incidence function $I(E,\mathcal{A})$. It is straightforward to observe that for every $A\in \mathcal{A}$, we have $I(E,\{A\})=|E\cap A|$. Consequently, equation \eqref{equation moment} represents the $r$th moment of the incidence function $I(E,\{A\})$ over all $A\in \mathcal{A}$. Murphy and Petridis \cite{murphy} computed the exact value of $\mathcal{E}_2(E, A(n-1,n))$, but their proof was purely combinatorial, relying on standard second-moment calculations. Chen \cite[Proposition 6.1]{chen} extended this result to all $1\leq k\leq n$ and gave a Fourier analytic proof in the case of prime $q$. In the following lemma, we record the exact value of $\mathcal{E}_2(E, A(k,n))$ for all $1\leq k\leq n-1$ and for vector spaces over arbitrary finite fields $\mathbb{F}_q$ (Lemma \ref{0-th moment}, identity \eqref{second moment}) as well as the 0th and 1st moments, which we will also need later. 

\begin{lma}
\label{0-th moment}
Let $k,n\in \mathbb{N}_0$ such that $0\leq k\leq n$. Then we have:
\begin{equation}
\label{zero moment}
    \mathcal{E}_0(E,A(k,n))=q^{n-k}\binom{n}{k}_q,
\end{equation}
\begin{equation}
\label{first moment}
    \mathcal{E}_1(E,A(k,n))=|E|\binom{n}{k}_q,
\end{equation}    
\begin{equation}
\label{second moment}
    \mathcal{E}_2(E,A(k,n))=|E|q^k\binom{n-1}{k}_q+|E|^2\binom{n-1}{k-1}_q.
\end{equation}    
\end{lma}

We defer the proof of Lemma \ref{0-th moment} to Section \ref{moment proof}. 

\subsection{Main results on incidences}

We now state our main theorem concerning incidence geometry.  Although the theorem itself is not new, to the best of our knowledge, our proof is novel. In particular, it avoids the use of spectral graph theory. Instead, our approach is direct, relying on elementary combinatorics and discrete Fourier analysis.  The proof is based on Lemma \ref{0-th moment}, which in turn depends on Lemma \ref{sum of characters} and Corollary \ref{Plancherel on subspaces}.

\begin{thm}
\label{incidence theorem}
Let $k,n\in \mathbb{N}_0$. If $E\subseteq\mathbb{F}_q^n$ and $\mathcal{A}\subseteq A(k,n)$, then
\begin{equation}
\label{main incidence bound}
    \left|I(E,\mathcal{A})-\frac{|E||\mathcal{A}|}{q^{n-k}}\right|\leq \left(q^{k(n-k)}| E ||\mathcal{A}|\right)^{\frac{1}{2}} \left(\frac{\binom{n-1}{k}_q\left(1-q^{-n}|E|\right)}{q^{k(n-k-1)}}\right)^{\frac{1}{2}},
\end{equation}
where we note that the fraction inside the larger brackets is bounded above by $1+o(1)$.
\end{thm}

We defer the proof of Theorem \ref{incidence theorem} to Section \ref{incidence proof}. Several useful consequences follow from Theorem \ref{incidence theorem}. First, note that this theorem quantifies the $(1+o(1))$ factor appearing in \eqref{vinh's higher dimensions for k}. However, we stress that this is \emph{not} new and, in fact, one can recover the same factor by following the proof in \cite{lund} or by consulting \cite{haemers}.  One of the first questions we asked ourselves was whether the $(1+o(1))$ factor is actually necessary, since it does not appear in Vinh's estimate \eqref{vinh's result} in the case $n=2$.  The next corollary identifies sufficient conditions for this factor to be removed.  It recovers Vinh's result  \eqref{vinh's result} and improves his point-hyperplane bound \eqref{vinh's higher dimensions} by removing the $(1+o(1))$ factor.

\begin{cor}
\label{cor 1}
Let $n\geq 2$. If either $k=n-1$ or $ q^{n-1}=o(|E|)$ and $q$ is large enough, then 
\[
    \left|I(E,\mathcal{A})-\frac{|E||\mathcal{A}|}{q^{n-k}}\right|\leq \left(q^{k(n-k)}| E ||\mathcal{A}|\right)^{\frac{1}{2}}.
\]
\end{cor}
\begin{proof}[Proof of Corollary \ref{cor 1}]
    If $k=n-1$, the result follows immediately since $\binom{n-1}{n-1}_q=1$. On the other hand, regardless of the value of $k$, if $ q^{n-1}=o(|E|)$, then 
\begin{align*}
\frac{\binom{n-1}{k}_q\left(1-q^{-n}|E|\right)}{q^{k(n-k-1)}}=(1+O(q^{-1}))(1-q^{-n}|E|) \leq 1
\end{align*}
for sufficiently large $q$.
\end{proof}
Finally, we state a corollary that explicitly describes how the bounds for $I(E, \mathcal{A})$ in Theorem \ref{incidence theorem} depend on the relative size of $|E||\mathcal{A}|$. It is worth noting that $|E|$ can be as large as $q^n$, while $|\mathcal{A}|$ can be as large as $ q^{(k+1)(n-k)}$.

\begin{cor}
\label{cor 2}
Let $n\geq 2$,  $1\leq k\leq n-1$,  $E\subseteq\mathbb{F}_q^n$, and $\mathcal{A}\subseteq A(k,n)$.
\begin{enumerate}
    \item If $|E||\mathcal{A}| \geq c q^{(k+2)(n-k)}$ for some $c>1$, then 
    \begin{equation}
    \label{size of EA is large}
    I(E,\mathcal{A})\approx \frac{|E||\mathcal{A}|}{q^{n-k}}    
    \end{equation}
     for sufficiently large $q$.
    \item If $c' q^{k(n-k)}<|E||\mathcal{A}|\leq  q^{(k+2)(n-k)}$ for some $c'>4$, then 
    \begin{equation}
    \label{size of EA is medium}
    I(E,\mathcal{A})\leq (2+o(1)) q^{^{\frac{k(n-k)}{2}}}|E|^{\frac{1}{2}}|\mathcal{A}|^{\frac{1}{2}},        \end{equation}
    which improves upon the trivial bound for sufficiently large $q$.
    \item If $|E||\mathcal{A}| \leq 4 q^{k(n-k)}$, then we cannot improve upon the trivial bound: 
    \begin{equation}
    \label{size of EA is small}
    I(E,\mathcal{A})\leq |E||\mathcal{A}|.        \end{equation}
\end{enumerate}
\end{cor}

\begin{proof}[Proof of Corollary \ref{cor 2}]
   Theorem \ref{incidence theorem} gives:
    \begin{equation}
    \label{upper and lower bound for incidence function}
        \frac{|E||\mathcal{A}|}{q^{n-k}}-(1+o(1))q^{^{\frac{k(n-k)}{2}}}|E|^{\frac{1}{2}}|\mathcal{A}|^{\frac{1}{2}} \leq I(E,\mathcal{A})\leq \frac{|E||\mathcal{A}|}{q^{n-k}}+(1+o(1))q^{^{\frac{k(n-k)}{2}}}|E|^{\frac{1}{2}}|\mathcal{A}|^{\frac{1}{2}}.
    \end{equation}

    \begin{enumerate}
        \item If $|E||\mathcal{A}|\geq cq^{(k+2)(n-k)}$ for some $c>1$, then it is easy to verify that
        \begin{equation*}
            \left(1-\frac{1+o(1)}{\sqrt{c}} \right)\frac{|E||\mathcal{A}|}{q^{n-k}}\leq I(E,\mathcal{A})\leq \left(1+\frac{1+o(1)}{\sqrt{c}} \right)\frac{|E||\mathcal{A}|}{q^{n-k}},
        \end{equation*}
        which implies \eqref{size of EA is large} since $c>1$.
        \item If $c'q^{k(n-k)}< |E||\mathcal{A}|\leq q^{(k+2)(n-k)}$ for some $c'>4$, then it is easy to verify that 
        \begin{equation*}
        I(E,\mathcal{A})\leq (2+o(1)) q^{\frac{k(n-k)}{2}}|E|^{\frac{1}{2}}|\mathcal{A}|^{\frac{1}{2}}  ,  
        \end{equation*}
        and one can check that the right-hand side of the above inequality provides a stronger bound than the trivial estimate for $I(E,\mathcal{A})$, which is $|E||\mathcal{A}|$. This establishes \eqref{size of EA is medium}. 
        \item There is nothing to prove in this case. \qedhere
    \end{enumerate}
\end{proof}

\subsection{Sharpness examples}

In this section, we demonstrate that Corollary \ref{cor 2}, and consequently Theorem \ref{incidence theorem}, is  sharp in the sense that, up to constants, one cannot improve the bounds for $I(E,\mathcal{A})$ as a function of the product $|E| |\mathcal{A}|$. First, observe that the result in case (1) of Corollary \ref{cor 2} is immediately seen to be sharp up to constants, as it provides an explicit formula for the number of incidences (up to constants).  This marks a notable distinction between incidence geometry in finite fields and in Euclidean space, where there is always enough room to position points and planes in such a way that they do not intersect at all, thereby prohibiting nontrivial lower bounds on the incidence function in general.

The fact that Theorem \ref{incidence theorem} is sharp regrettably demonstrates that certain  incidence results of this type appearing in the literature are incorrect.  In particular, \cite[Theorem 2.1]{bennett} claims that 
\[
I(E,\mathcal{A})\leq \frac{|E||\mathcal{A}|}{q^{n-k}}+\left(q^{n+k-2}|E||\mathcal{A}|\right)^{\frac{1}{2}} (1+o(1))
\]
when $k<n-1$.  We show that this  estimate is not generally true whenever $n+k-2<k(n-k)$, and in particular, whenever $n \geq 5$ and $2 \leq k \leq n-3$.  Indeed, the estimate already fails for $n=5,k=2$ in the simple case where $E=\{0\}$ is a single point and $\mathcal{A}$ consists of all 2-dimensional linear subspaces, i.e., $\mathcal{A}=G(2,5)$.  The number of incidences is at least $q^{6}$, whereas \cite[Theorem 2.1]{bennett} predicts that this number can be bounded by a constant multiple of $q^{\frac{11}{2}}$.

\subsubsection{Examples with few incidences}

Consider the threshold between cases (1) and (2) in Corollary \ref{cor 2}. Let $E \subseteq \mathbb{F}_q^n$ be a set with
\begin{equation*}
   \frac{q^{n-k}}{2} < |E| \leq  \frac{q^{n-k}}{2}    +1
\end{equation*}
and define
\[
\mathcal{A} \coloneqq \{A \in A(k,n) : A \cap E = \varnothing\}.
\]
Then $I(E,\mathcal{A})=0$ by construction. Since each point $x \in \mathbb{F}_q^n$ can intersect at most $|G(k,n)|$ elements $A \in A(k,n)$, we obtain
\begin{align*}
|\mathcal{A}| &\geq |A(k,n)|-|E||G(k,n)| \\
&\geq  |G(k,n)|\left(q^{n-k}-\frac{q^{n-k}}{2}-1\right) \\
&\geq \frac{q^{(k+1)(n-k)}}{2} - q^{k(n-k)}.
\end{align*}
Therefore,
\begin{align*}
|E||\mathcal{A}| &>  \frac{q^{n-k}}{2} \left(\frac{q^{(k+1)(n-k)}}{2} - q^{k(n-k)}\right) \\
&= \frac{q^{(k+2)(n-k)}}{4} - \frac{q^{(k+1)(n-k)}}{2}.
\end{align*}
In particular, the condition $|E||\mathcal{A}| \geq c q^{(k+2)(n-k)}$ with $c<\frac{1}{4}$ is \emph{not} sufficient to ensure a nontrivial lower bound for the incidence function.  This demonstrates that our threshold $q^{(k+2)(n-k)}$ from Corollary \ref{cor 2} is sharp up to constant factors.

Again considering the threshold between cases (1) and (2), we provide another sharp example in the case $n=2$, which slightly outperforms the previous example. Let $k=1$, and let $K \subseteq \mathbb{F}_q^n$ be a Kakeya set with $|K| \sim 2^{1-n} q^n$.  (Such ``small Kakeya sets'' are well-known to exist; see \cite{dvir2}.) Define $E=\mathbb{F}_q^n \setminus K$, and let $\mathcal{A}$ be a collection of lines inside $K$ of size $|G(1,n)| \sim q^{n-1}$ (such a collection exists by the definition of a Kakeya set).  Then
\begin{equation*}
|E||\mathcal{A}| \sim  (1-2^{1-n})q^{2n-1},
\end{equation*}
but $I(E,\mathcal{A})=0$. In particular, the condition $|E||\mathcal{A}| \geq c  (1-2^{1-n})q^{2n-1}$ with $0<c<1$ is \emph{not} sufficient to ensure a  nontrivial lower bound for the incidence function. When $n=2$, this matches our threshold $|E||\mathcal{A}| =   q^{3(n-1)}$ up to a constant factor of 2. However, for $n>2$, it does not reach the threshold.

\subsubsection{Examples with many incidences}

We have so far demonstrated that the estimates in case (1) are sharp and that the threshold between cases (1) and (2) is sharp.  It remains to show that the estimate in case (2) is also sharp; that is, for all values of $|E||\mathcal{A}|$ in the relevant range, we can construct suitable sets $E$ and $\mathcal{A}$ where the upper bound is attained (up to constants).  

Let $E \subseteq \mathbb{F}_q^n$ satisfy $1 \leq |E| \leq q^{n-k}$. Define $\mathcal{A} \subseteq A(k,n)$ by
\[
\mathcal{A}= \{ x+V : x \in E, \, V \in G(k,n)\}.
\]
Then 
\begin{equation} \label{iea}
I(E, \mathcal{A}) \geq |E| |G(k,n)| \geq |E| q^{k(n-k)}
\end{equation}
and
\[
 |\mathcal{A}| \leq |E| |G(k,n)| =  |E|q^{k(n-k)}(1 +o(1)).
\]
By Theorem \ref{incidence theorem}, we obtain
\begin{align*}
I(E, \mathcal{A}) &\leq \frac{|E| |\mathcal{A}|}{q^{n-k}}+ \  q^{^{\frac{k(n-k)}{2}}}|E|^{\frac{1}{2}}|\mathcal{A}|^{\frac{1}{2}}(1 +o(1))\\
&\leq |E|q^{k(n-k)} \left( |E|q^{k-n}+ 1  \right)(1 +o(1))\\
&\leq |E|q^{k(n-k)} (2 +o(1)).
\end{align*}
Comparing with \eqref{iea} we see that this bound is attained up to a factor of $(2+o(1))$. Incidentally, this also shows that
\[
 |\mathcal{A}|  \geq  \frac{|E|q^{k(n-k)}}{4+o(1)},
\]
so we can choose $E$ such that  $|E||\mathcal{A}| \approx q^\alpha$ for any $\alpha$ satisfying 
\[
 k(n-k) \leq \alpha \leq  (k+2)(n-k),
\]
as desired.  Furthermore, in the case $|E| =o(q^{n-k})$, the estimate from Theorem \ref{incidence theorem} gives
\[
I(E, \mathcal{A}) \leq |E|q^{k(n-k)} ( 1+o(1))(1  +o(1)) \sim |E|q^{k(n-k)},
\]
which confirms that Theorem \ref{incidence theorem} provides the correct asymptotic value for the number of incidences.

\section{Incidence geometry: proof of Theorem \ref{incidence theorem}} \label{incidence proof}

\subsection{Moment identities: proof of Lemma \ref{0-th moment}} \label{moment proof}

\begin{enumerate}
   \item[$(r=0)$] By the definition of $\mathcal{E}_r$ for $r=0$ in $\eqref{equation moment}$, we have:
    \begin{equation*}
        \mathcal{E}_0(E,A(k,n))=\sum_{A\in A(k,n)}1=|A(k,n)|,
    \end{equation*}
    thus, it suffices to determine the cardinality of $A(k,n)$. If $V\in G(k,n)$, then $V$ has $q^{n-k}$ distinct cosets in $\mathbb{F}_q^n$. Suppose these cosets are given by $\{V+t_{V,j}:1\leq j\leq q^{n-k}\}$. Thus we have
    \begin{equation}
    \label{precise description of A(k,n)}
        A(k,n)=\{V+t_{V,j}: 1\leq j\leq q^{n-k}, V\in G(k,n)\}.
    \end{equation}
    We will make the following simple observation: if $(V,i)\neq (W,j)$, then $V+t_{V,i}\neq W+t_{W,j}$ for $1\leq i,j\leq q^{n-k}$ and $V,W\in G(k,n)$. Indeed, suppose that $V=W$. Then $i\neq j$, and hence
    $V+t_{V,i}\neq V+t_{V,j}$ because these cosets are distinct as $i\neq j$. Now suppose that $V\neq W$ but $V+t_{V,i}=W+t_{W,j}$. Then $V=W+\mathbf{x}$ for some $\mathbf{x}\in \mathbb{F}_q^n$. This implies that $\mathbf{x}=-\mathbf{w}$ for some $\mathbf{w}\in W$. Therefore, $V=W+(-\mathbf{w})=W$, which is a contradiction since we assumed that $V\neq W$.    This observation immediately implies that
    \begin{equation*}
        |A(k,n)|=q^{n-k} |G(k,n)|=q^{n-k} \binom{n}{k}_q,
    \end{equation*}
    which completes the proof of \eqref{zero moment}.\\

\item[$(r=1)$] By the definition of $\mathcal{E}_r$ for $r=1$ in $\eqref{equation moment}$, we have:
    \begin{equation*}
        \mathcal{E}_1(E,A(k,n))=\sum_{A\in A(k,n)}|E\cap A|.
    \end{equation*}    
Writing $E(x)$ and $A(x)$ for the indicator functions of $E$ and $A$, and changing the order of summation, we obtain:
\begin{align*}
    \sum_{A\in A(k,n)}|E\cap A|&=\sum_{A\in A(k,n)}\sum_{x\in A}E(x) =\sum_{x\in E}\sum_{A\in A(k,n)}A(x).
\end{align*}
Now, using \eqref{precise description of A(k,n)}, for every $x\in \mathbb{F}_q^n$, we have:
\begin{align*}
    \sum_{A\in A(k,n)}A(x)&=\sum_{V\in G(k,n)}\sum_{\substack{j=1 \\ x\in V+t_{V,j}}}^{q^{n-k}}1 =\sum_{V\in G(k,n)}1=\binom{n}{k}_q.
\end{align*}
Therefore, we obtain:
\begin{align*}
    \mathcal{E}_1(E,A(k,n))=|E|\binom{n}{k}_q,
\end{align*}
which completes the proof of \eqref{first moment}.\\

\item[$(r=2)$] This case follows Chen's Fourier analytic proof, which was applied in the case of prime $q$ \cite{chen}. By the definition of $\mathcal{E}_r$ for $r=2$ in \eqref{equation moment} and \eqref{precise description of A(k,n)}, we have:
\begin{align}
    \mathcal{E}_2(E,A(k,n))&=\sum_{A\in A(k,n)}|E\cap A|^2  =\sum_{V\in G(k,n)}\sum_{j=1}^{q^{n-k}}|E\cap (V+t_{V,j})|^2 \label{double sum for 2nd moment} .
\end{align}
Applying Corollary \ref{Plancherel on subspaces} to \eqref{double sum for 2nd moment}, where $n-k$ should be replaced by $k$, we obtain:
\begin{align}
    \mathcal{E}_2(E,A(k,n))&=q^{n+k}\sum_{V\in G(k,n)}\sum_{m\in V^{\perp}}|\widehat{E}(m)|^2 \nonumber \\
    &=q^{n+k}\left(\sum_{V\in G(k,n)}\sum_{m\in V^{\perp}\setminus \{0\}}|\widehat{E}(m)|^2+|G(k,n)|\cdot|\widehat{E}(0)|^2\right) \nonumber \\
    &=q^{n+k}\left(\sum_{m\in \mathbb{F}_q^n\setminus \{0\}}|\widehat{E}(m)|^2\cdot |\{V\in G(k,n):m\in V^{\perp}\}|+|E|^2q^{-2n}\binom{n}{k}_q \right) \nonumber \\
    &=q^{n+k}\left(\binom{n-1}{k}_q\sum_{m\in \mathbb{F}_q^n\setminus \{0\}}|\widehat{E}(m)|^2+|E|^2q^{-2n}\binom{n}{k}_q \right) \label{app of Gauss 1},
\end{align}
where in \eqref{app of Gauss 1}, we have used \eqref{size of G(k,n):ort(V) contains z}. Moreover, using \eqref{Plancherel's theorem for E}, we have:
\begin{align*}
    \mathcal{E}_2(E,A(k,n))&=q^{n+k}\left( \binom{n-1}{k}_q (|E|q^{-n}-|E|^2q^{-2n})+|E|^2q^{-2n}\binom{n}{k}_q\right)\\
    &=|E|q^k \binom{n-1}{k}_q+|E|^2q^{-n+k} \left(\binom{n}{k}_q-\binom{n-1}{k}_q\right)\\
    &=|E|q^k \binom{n-1}{k}_q+|E|^2\binom{n-1}{k-1}_q, 
\end{align*}
where in the final line, we have used \eqref{1st Pascal's id}, thereby proving the desired identity.
\end{enumerate}

\subsection{Proof of Theorem \ref{incidence theorem}} 

By applying the triangle inequality and the Cauchy-Schwarz inequality, we obtain
\begin{align}
    \left|I(E,\mathcal{A})-\frac{|E||\mathcal{A}|}{q^{n-k}}\right|&=\left|\sum_{A\in \mathcal{A}}\left(|E\cap A|-\frac{|E|}{q^{n-k}}\right) \right| \nonumber \leq \sum_{A\in \mathcal{A}}\left||E\cap A|-\frac{|E|}{q^{n-k}}\right| \nonumber \\
    &\leq |\mathcal{A}|^{\frac{1}{2}}\left(\sum_{A\in \mathcal{A}}\left(|E\cap A|-\frac{|E|}{q^{n-k}}\right)^2\right)^{\frac{1}{2}} \nonumber \\
    &\leq |\mathcal{A}|^{\frac{1}{2}}\left(\sum_{A\in A(k,n)}\left(|E\cap A|-\frac{|E|}{q^{n-k}}\right)^2\right)^{\frac{1}{2}} \label{ineq for incidence}.
\end{align}

The sum in \eqref{ineq for incidence} can be expressed in terms of $\mathcal{E}_r(E,A(k,n))$ for $r\in \{0,1,2\}$ as follows:
\begin{align*}
&\sum_{A\in A(k,n)}\left(|E\cap A|-\frac{|E|}{q^{n-k}}\right)^2\\
&=\sum_{A\in A(k,n)}|E\cap A|^2-\frac{2|E|}{q^{n-k}}\sum_{A\in A(k,n)}|E\cap A|+\frac{|E|^2}{q^{2(n-k)}}\sum_{A\in A(k,n)}1\\
&=\mathcal{E}_2(E,A(k,n))-\frac{2|E|}{q^{n-k}}\mathcal{E}_1(E,A(k,n))+\frac{|E|^2}{q^{2(n-k)}}\mathcal{E}_0(E,A(k,n)).
\end{align*}

Substituting the values of $\mathcal{E}_r(E,A(k,n))$ for $r\in \{0,1,2\}$ from Lemma \ref{0-th moment}, we obtain:
\begin{align}
\sum_{A\in A(k,n)}  \left(|E\cap A| -\frac{|E|}{q^{n-k}}\right)^2&=|E|q^k \binom{n-1}{k}_q+|E|^2\binom{n-1}{k-1}_q-\frac{|E|^2}{q^{n-k}}\binom{n}{k}_q \nonumber \\
&=|E|q^k \binom{n-1}{k}_q+|E|^2q^{k-n}\left(q^{n-k}\binom{n-1}{k-1}_q-\binom{n}{k}_q\right) \nonumber \\
&=|E|q^k \binom{n-1}{k}_q-|E|^2q^{k-n}\binom{n-1}{k}_q, \label{upper bound for 2nd moment of deviation}
\end{align}
where \eqref{1st Pascal's id} has been used in the last line.  Combining \eqref{ineq for incidence} with \eqref{upper bound for 2nd moment of deviation} and performing some arithmetic operations, we obtain the following inequality:
\begin{align*}
    \left|I(E,\mathcal{A})-\frac{|E||\mathcal{A}|}{q^{n-k}}\right|\leq\left(q^{k(n-k)}| E ||\mathcal{A}|\right)^{\frac{1}{2}} \left(\frac{\binom{n-1}{k}_q\left(1-q^{-n}|E|\right)}{q^{k(n-k-1)}}\right)^{\frac{1}{2}},
\end{align*}
which establishes \eqref{main incidence bound}.

\section{Bringing it all back home: from incidences back to projections} \label{biabh}

In this final section, we come full circle and reconnect our work on incidences with our work on projections.  In fact, we provide an alternative and direct proof of Corollary \ref{main cor} as a simple application of Theorem \ref{incidence theorem}.  As a result,  we now have four  proofs of the finite field analogue of Marstrand's projection theorem \eqref{marstrandfinite}, albeit with some substantial overlaps:  
\begin{enumerate}
\item the special case of \cite[Theorem 1.7]{bright}, which yields \cite[(13)]{bright};
    \item setting $p=2$ in Theorem \ref{main projection theorem};
    \item applying Lemma \ref{sum of characters} together with the argument of Chen \cite{chen};
    \item the proof we present below, which is based on incidence geometry.
\end{enumerate}

\begin{cor} 
\label{projfromit}
Let $E\subseteq \mathbb{F}_q^n$ be a nonempty set and $0<u\leq\frac{q^k}{2}$. Then
\begin{equation}
\label{discrete Marstrand new}
    |\{V\in G(k,n): |\pi_V(E)|\leq u\}|\leq (4+o(1)) uq^{k(n-k)}|E|^{-1}.
\end{equation}
\end{cor}

\begin{proof}[Proof of Corollary \ref{projfromit}]
    Define  
\begin{equation*}
  \Theta \coloneqq \{V\in G(k,n): |\pi_V(E)|\leq u\}, 
\end{equation*}
and let $\mathcal{A} \subseteq A(n-k,n)$ be given by
\begin{equation}
\label{definioon of A}
    \mathcal{A}\coloneqq\bigcup_{V \in \Theta}\pi_V(E),
\end{equation}
recalling Definition \ref{projdef}. We observe that the union in \eqref{definioon of A} is disjoint. Indeed, suppose it is not disjoint; then there exist distinct $V_1, V_2\in \Theta$ such that $x+V_1^{\perp}=y+V_2^{\perp}$ for some $x,y \in \mathbb{F}_q^n$. This implies that $z+V_1^{\perp}=V_2^{\perp}$ for $z=x-y$. Since $0\in V_2^{\perp}$, there exists $w\in V_1^{\perp}$ such that $z+w=0$, meaning $z=-w\in V_1^{\perp}$. Consequently, we have $V_2^{\perp}=z+V_1^{\perp}=V_1^{\perp}$, and hence $V_1=V_2$, leading to a contradiction.

We need the following useful estimates.
\begin{lma}
\label{properties of A}
The following inequalities hold:
    \begin{enumerate}
    \item $|\mathcal{A} | \leq u |\Theta|$,
    \item $I(E,\mathcal{A}) \geq |\Theta| |E|$.
\end{enumerate}
\end{lma}
\begin{proof}[Proof of Lemma \ref{properties of A}]
\
\begin{enumerate}
    \item From \eqref{definioon of A}, we have:
    \begin{align*}
        |\mathcal{A}|&\leq \sum_{V\in \Theta}|\pi_V(E)|\leq \sum_{V\in \Theta}u=u|\Theta|.
    \end{align*}
    \item We use the disjointness of the union in \eqref{definioon of A}, established above, to prove the second claim:
    \begin{align*}
        I(E,\mathcal{A})=\sum_{V\in \Theta}I(E,\pi_V(E))=\sum_{V\in \Theta}\sum_{p\in E}|\{A\in \pi_V(E): p\in A\}|\geq |\Theta||E|,
    \end{align*}
    where we have used the fact that for any fixed $V\in \Theta$ and $p\in E$, we have $|\{A\in \pi_V(E): p\in A\}|\geq 1$. This follows because if $V\in \Theta$, then $V^{\perp}$ has $q^k$ distinct cosets of the form $\{t_j+V^{\perp}\}_{j=1}^{q^k}$. Since $E=\sqcup_{j=1}^{q^k}E\cap (t_j+V^{\perp})$, there exists some $j$ such that $p\in E\cap (t_j+V^{\perp})$. \qedhere
\end{enumerate}
\end{proof}

Applying Theorem \ref{incidence theorem}, Lemma \ref{properties of A}, and given that $0<u\leq \frac{q^k}{2}$, we obtain
 \begin{align*}
 |\Theta||E| &\leq I(E,\mathcal{A}) \leq \frac{|E| |\mathcal{A}|}{q^k} +(1+o(1)) q^{\frac{k(n-k)}{2}}|E|^{\frac{1}{2}}|\mathcal{A}|^{\frac{1}{2}}\\
 &\leq \frac{u|E||\Theta|}{q^k} +(1+o(1)) u^{\frac{1}{2}}q^{\frac{k(n-k)}{2}}|E|^{\frac{1}{2}}|\Theta|^{\frac{1}{2}}\\
 &\leq\frac{|E||\Theta|}{2} +(1+o(1)) u^{\frac{1}{2}}q^{\frac{k(n-k)}{2}}|E|^{\frac{1}{2}}|\Theta|^{\frac{1}{2}}.
 \end{align*}
 Rearranging terms, we obtain
\begin{equation*}
\frac{|\Theta||E|}{2}\leq(1+o(1)) u^{\frac{1}{2}}q^{\frac{k(n-k)}{2}}|E|^{\frac{1}{2}}|\Theta|^{\frac{1}{2}}.
\end{equation*} 
Squaring both sides and simplifying, we conclude
\begin{equation*}
|\Theta| \leq (4+o(1)) uq^{k(n-k)}|E|^{-1},    
\end{equation*}
as required.
\end{proof}

\section*{Acknowledgments}

We are grateful to Alex Iosevich, Alex McDonald, Ana de Orellana, and Steve Senger for their helpful comments and discussions.


\begin{thebibliography}{BGGIST07}

\bibitem[BIP14]{bennett}
M. Bennett, A. Iosevich, and J. Pakianathan. Three-point configurations determined by subsets of $\mathbb{F}_q^2$ via the elekes-sharir paradigm, \emph{Combinatorica}, {\bf 34}, (2014), 689--706.

\bibitem[BG23+]{bright}
P. Bright and S. Gan.
 Exceptional set estimates in finite fields
 preprint:	\href{https://arxiv.org/abs/2302.13193}{arXiv:2302.13193} (2023). 


\bibitem[Che18]{chen}
C. Chen. Projections in vector spaces over finite fields, \emph{Ann. Acad. Sci. Fenn. Math.}, {\bf 43},  171-185, (2018).


 \bibitem[DDL21]{dvirda}
M. Dhar,  Z. Dvir and B. Lund.
Simple Proofs for Furstenberg Sets Over
Finite Fields, \emph{Discrete Analysis}, {\bf 22}, (2021), 16 pp.



 \bibitem[Dvi09]{dvir}
Z. Dvir.  On the size of Kakeya sets in finite fields. \emph{J. Amer. Math. Soc.}, {\bf  22},  (2009),   1093--1097.  


\bibitem[FdO24+]{ana}
J. M. Fraser and  A. E. de Orellana.
 A Fourier analytic approach to exceptional set estimates for orthogonal projections, \emph{Indiana Univ. Math. J.}, (to appear), available at: \href{https://arxiv.org/abs/2404.11179}{arXiv:2404.11179}
 
\bibitem[Fra24+]{fraserfinite}
J. M. Fraser. $L^p$ averages of the Fourier transform in finite fields. preprint:	\href{https://arxiv.org/abs/2407.08589}{arXiv:2407.08589} (2024). 

\bibitem[Hae80]{haemers}
W. H.  Haemers. Eigenvalue techniques in design and graph theory, 121, Mathematisch centrum Amsterdam, 1980.


\bibitem[IR07]{iosevich}
A. Iosevich and  M. Rudnev.
Erd{\H{o}}s distance problem in vector spaces over finite fields, 
\emph{Trans. Amer. Math. Soc.}, {\bf  359},  (2007),   6127--6142.





\bibitem[LN97]{lidl}
R. Lidl and H. Niederreiter.
\emph{Finite fields},   Second edition. Encyclopedia of Mathematics and its Applications, {\bf 20}, (Cambridge University Press,  1997).


\bibitem[LPV23+]{lundphamvinh}
B.  Lund, T. Pham, L. A. Vinh. Orthogonal projections in the plane over prime order fields. preprint:	\href{https://arxiv.org/abs/2311.05148}{arXiv:2311.05148} (2024). 
 



\bibitem[LS16]{lund}
B.  Lund and S. Saraf. Incidence bounds for block designs, \emph{SIAM Journal on Discrete
Mathematics}, {\bf 30}, (2016), 1997--2010.


   \bibitem[Mat15]{mattila}P. Mattila. \emph{Fourier analysis and Hausdorff dimension}. Cambridge Studies in Advanced Mathematics, \textbf{150}, Cambridge, (2015).


  \bibitem[Mat75]{mattila75}P. Mattila. Hausdorff dimension, orthogonal projections and intersections with planes. \emph{Ann. Fenn. Math.}, \textbf{1}(2), 227--244. 

\bibitem[MP16]{murphy}
 B. Murphy and G. Petridis, A point-line incidence identity in finite fields, and applications, \emph{Mosc. J. Comb. Number Theory}, {\bf 6}, (2016), no.~1, 64--95.
  
  \bibitem[PS00]{peres}Y. Peres and W. Schlag. Smoothness of projections, Bernoulli convolutions, and the dimension of exceptions. \emph{Duke Math. J.}, \textbf{102}, (2000), 193--251.

 \bibitem[SS08+]{dvir2}S. Saraf and M. Sudan.
Improved lower bound on the size of Kakeya sets over finite fields, https://arxiv.org/pdf/0808.2499

\bibitem[ST83]{szem}  
E. Szemer\'edi and W.~T. Trotter Jr., Extremal problems in discrete geometry, \emph{Combinatorica}, {\bf 3}, (1983), no.~3-4, 381--392.

\bibitem[Vin11]{vinh}
 L.~A.~Vinh, The Szemer\'edi-Trotter type theorem and the sum-product estimate in finite fields, \emph{European J. Combin.}, {\bf 32}, (2011), no.~8, 1177--1181.

\bibitem[Wol99]{wolff}
 T. Wolff. Recent work connected with the Kakeya problem. \emph{Prospects in mathematics (Princeton, NJ, 1996)}, 129--162, Amer. Math. Soc., Providence, RI, (1999).


\end{thebibliography}
\end{document}